\documentclass[11pt]{article}

\usepackage{latexsym}
\usepackage{amsmath}
\usepackage{amssymb}
\usepackage{amsthm}
\usepackage{hyperref}
\usepackage{cite}
\usepackage{graphicx}
\usepackage{color}
\usepackage{enumitem}
\usepackage{braket}
\usepackage{bm}
\usepackage[left=1in,right=1in,top=1in,bottom=1in]{geometry}
\usepackage{xspace}
\setlist[description]{font=\normalfont\itshape\textbullet\space}
\usepackage[all]{xy}

\renewcommand{\paragraph}[1]{\vspace{6pt} \noindent \textbf{#1}\xspace}

\theoremstyle{plain}
\newtheorem{theorem}{Theorem}[section]

\newtheorem{lemma}[theorem]{Lemma}
\newtheorem{observation}[theorem]{Observation}
\newtheorem{proposition}[theorem]{Proposition}
\newtheorem{claim}[theorem]{Claim}
\newtheorem{fact}[theorem]{Fact}

  {\end{tabular}\par\medskip}

\theoremstyle{definition}

\newtheorem{definition}[theorem]{Definition}

\newcommand{\GL}{\mathrm{GL}}
\newcommand{\F}{\mathbb{F}}

\newcommand{\D}{\mathrm{D}}
\newcommand{\cat}[1]{\mathsf{#1}}

\newcommand{\catgraphpull}{\cat{Graph}_{\leftarrow}}
\newcommand{\BLTfunctor}{\mathsf{BLT}}
\newcommand{\catgroup}{\cat{Group}}

\newcommand{\N}{\mathbb{N}}

\newcommand{\rk}{\mathrm{rk}}

\newcommand{\im}{\mathrm{im}}

\newcommand{\sym}{\mathrm{Sym}}

\newcommand{\M}{\mathrm{M}}

\renewcommand{\S}{\mathrm{S}}

\newcommand{\spa}[1]{\mathcal{#1}}

\newcommand{\cA}{\spa{A}}
\newcommand{\cB}{\spa{B}}

\newcommand{\bK}{\mathbf{K}}

\newcommand{\bA}{\mathbf{A}}

\renewcommand{\ge}{\geqslant}

\newcommand{\too}%
{\xrightarrow{\text{\raisebox{-3pt}{$\sim$}}\,}}

\usepackage{arydshln}
\hypersetup{
    bookmarksnumbered=true,     
    bookmarksopen=true,         
    bookmarksopenlevel=1,       
    colorlinks=true,            
    pdfstartview=Fit,           
    pdfpagemode=UseOutlines,    % this is the option you were lookin for
    pdfpagelayout=OneColumn,
    pdfstartview=FitH
}
\title{
On the Baer-Lov\'asz-Tutte construction 
of groups from graphs: isomorphism types and homomorphism notions
}

\author{
Xiaoyu He
\footnote{Institute of Computing Technology, Chinese Academy of Sciences, China, 
and
University of Chinese Academy of Sciences, China. \tt{hexiaoyu18s@ict.ac.cn}}
\and
Youming Qiao
\footnote{Centre for Quantum Software and Information, University of 
Technology Sydney. \tt{jimmyqiao86@gmail.com}}
}

\date{\today}

\begin{document}

%\pagenumbering{gobble}  % no page number on first page

\maketitle

\begin{abstract}
Let $p$ be an odd prime. From a simple undirected 
graph $G$, through the classical procedures of Baer (Trans. Am. Math. Soc., 1938), 
Tutte (J. Lond. Math. Soc., 1947) and Lov\'asz (B. Braz. Math. Soc., 1989), there 
is a $p$-group $P_G$ of class $2$ and exponent $p$ that is naturally associated 
with $G$. Our first result is to show that this construction of groups from graphs 
respects 
isomorphism types. That is, 
given two graphs $G$ and $H$, $G$ and $H$ are isomorphic as graphs if and only if 
$P_G$ and $P_H$ are isomorphic as groups. Our second contribution is a new 
homomorphism notion for graphs. Based on this notion, a category of 
graphs can be defined, and the 
Baer-Lov\'asz-Tutte construction naturally leads to a functor from this category 
of graphs to the category of groups. 
\end{abstract}

\section{Introduction}

\subsection{The results}

In this note we study some basic questions regarding the following construction of 
finite groups from simple undirected graphs, following the classical works of 
Baer \cite{Bae38}, Tutte \cite{Tut47}, and Lov\'asz \cite{Lov79}.

\paragraph{Notations.} To introduce this construction, we set up some notations 
first. 
%For a finite set $V$, we use $\binom{V}{2}$ to denote the set of size-$2$ 
%subsets of $V$. We shall restrict to consider simple undirected graphs. A graph 
%$G=(V, E)$ consists of the vertex set $V$ and the edge set $E\subseteq 
%\binom{V}{2}$. 
For $n\in\N$, let 
$[n]=\{1, \dots, n\}$. The set of size-$2$ subsets of $[n]$ is denoted as 
$\binom{[n]}{2}$. The natural total order of $[n]$ induces the lexicographic 
order 
on $[n]\times [n]$. In this introduction, to ease the exposition, we shall 
only 
consider 
graphs with vertext sets being $[n]$. Therefore, a 
simple undirected graph with the vertex set $[n]$ is a subset of 
$\binom{[n]}{2}$. 
%
%Let $\F$ be a field. Let $V$ be a finite set. We 
%use $\F^V$ to denote the $\F$-vector space of functions from $V$ to $\F$. An 
%alternating bilinear form on $\F^V$ is 

For a field $\F$, $\F^n$ is the linear space consisting of length-$n$ 
column vectors over $\F$. We use $\M(\ell\times n, \F)$ to denote the linear space 
of $\ell 
\times n$ 
matrices over $\F$, and set $\M(n, \F):=\M(n\times n, \F)$. A matrix $A\in \M(n, 
\F)$ 
is alternating, if for any $v\in \F^n$, $v^tAv=0$. For 
$\{i,j\}\in \binom{[n]}{2}$, $i<j$, an elementary alternating matrix $A_{i,j}\in 
\M(n, \F)$ is the $n\times n$ matrix with the $(i,j)$th entry being $1$, the 
$(j,i)$th entry being $-1$, and the rest entries being $0$. The linear space of 
$n\times n$ alternating matrices over $\F$ is denoted by $\Lambda(n, \F)$. The 
general linear group of degree $n$ over $\F$ is denoted by $\GL(n, 
\F)$. 

\paragraph{The Baer-Lov\'asz-Tutte construction.} We now introduce the 
Baer-Lov\'asz-Tutte construction of groups from 
graphs with vertex sets being $[n]$. See Section~\ref{sec:hom} for a description 
of the construction that works for any finite set 
without specifying some total order.
\begin{enumerate}
\item\label{item:step1} In the first step, we construct an alternating 
bilinear map from a graph, 
following ideas traced back to Tutte \cite{Tut47} and Lov\'asz \cite{Lov79}. 
This steps 
works for any 
field $\F$. Let $G\subseteq \binom{[n]}{2}$ be a graph of size $m$. Suppose 
$G=\{\{i_1, j_1\}, \dots, \{i_m, j_m\}\}$, where for $k\in[m]$, $i_k<j_k$ and for 
$1\leq k<k'\leq m$, $(i_k, j_k)$ is less than $(i_{k'}, j_{k'})$ in the 
lexicographic order. For $k\in[m]$, let $A_k$ be the $n\times n$ elementary 
alternating matrix 
$A_{i_k, j_k}$ over $\F$, and set 
$\bA_G=(A_1, \dots, A_m)\in\Lambda(n, \F)^m$. Then $\bA_G$ defines an alternating 
bilinear map 
$\phi_G:\F^n\times\F^n\to\F^m$, by $\phi_G(v, u)=(v^tA_1u, \dots, v^tA_mu)^t$.
\item\label{item:step2} In the second step, we fix an odd prime $p$, take 
$\phi_G:\F_p^n\times 
\F_p^n\to\F_p^m$ from the last step over $\F_p$, and apply Baer's construction 
\cite{Bae38}
to $\phi_G$ to obtain $P_G$ which is a $p$-group of class $2$ and exponent $p$. 
Baer's 
construction actually works for any 
alternating bilinear map $\phi$ over $\F_p$. Specifically, let 
$\phi$ be 
an alternating bilinear map 
$\phi:\F_p^n\times\F_p^n\to\F_p^m$. A $p$-group of class $2$ and exponent $p$, 
denoted as $P_\phi$, can be constructed from $\phi$ as follows. The set of group 
elements of $P_\phi$ is $\F_p^n\times\F_p^m$. The group operation $\circ$ of 
$P_\phi$ is 
\begin{equation}\label{eq:gp_def}
(v_1, u_1)\circ (v_2, u_2) := (v_1+v_2, u_1+u_2+\frac{1}{2}\cdot \phi(v_1, v_2)).
\end{equation}
\end{enumerate}

%let $\cA_G$ be the linear span of 
%$\{A_{i,j} : \{i,j\}\in G\}$, called the alternating matrix space associated with 
%$G$. Take an ordered basis $\bA=(A_1, \dots, A_m)$ of $\cA_G$, which defined an 
%alternating bilinear map $\phi_\bA:\F^n\times\F^n\to\F^m$, by $\phi_\bA(v, 
%u)=(v^tA_1u, \dots, v^tA_mu)^t$. 

%Combining the above two steps, we can construct a $p$-group of class $2$ and 
%exponent $p$ from a graph $G$. We shall refer to the above procedure as the 
%Baer-Lo\'vasz-Tutte construction. 
We shall review the literature related to this construction in 
Section~\ref{subsec:related}. In this note we study two natural questions 
regarding isomorphisms and homomorphisms in this construction.

\paragraph{A result about isomorphisms.} Given two graphs $G, 
H\subseteq\binom{[n]}{2}$, the Baer-Lov\'asz-Tutte construction produce two 
$p$-groups of class $2$ and exponent $p$, 
$P_G$ and $P_H$. Clearly, if $G$ and $H$ are isomorphic, then $P_G$ and $P_H$ are 
isomorphic. Interestingly, the converse also holds. 
\begin{theorem}\label{thm:iso}
%Let $G$ and $H$ be two graphs, and let $P_G$ and $P_H$ be $p$-groups of class $2$ 
%and exponent $p$ constructed from $G$ and $H$ by the Baer-Lov\'asz-Tutte 
%procedure. 
Let $G$, $H$, $P_G$, and $P_H$ be as above. 
Then $G$ and $H$ are isomorphic as graphs if and only if $P_G$ and 
$P_H$ are isomorphic as groups. 
\end{theorem}

To prove Theorem~\ref{thm:iso}, we review the 
notion of 
isomorphisms for alternating bilinear maps. 
Let $\phi, \psi:\F^n\times\F^n\to\F^m$ be two alternating bilinear maps. Then 
$\phi$ and $\psi$ are \emph{isomorphic}\footnote{This isomorphism notion of 
alternating 
bilinear maps is sometimes called pseudo-isometry in the literature 
\cite{Wil09a}.}, if there exist $C\in\GL(n, \F)$ and $D\in\GL(m, 
\F)$, such that for any $u, v\in \F^n$, we have $D(\phi(u, v))=\psi(C(u), C(v))$. 
Let $\phi$, $\psi$, $P_\phi$ and $P_\psi$ be from Step~\ref{item:step2}. It is 
well-known that $\phi$ and $\psi$ are isomorphic if and only if $P_\phi$ and 
$P_\psi$ are isomorphic
(see e.g. \cite{Wil09a}). 

%That is, given two alternating bilinear maps $\phi, 
%\psi:\F_p^n\times\F_p^n\to\F_p^m$, construct $P_\phi$ and $P_\psi$ using Baer's 
%construction. Then we have $\phi$ and $\psi$ are isomorphic if and only if 
%$P_\phi$ and $P_\psi$ are isomorphic. 

Therefore, the crux of Theorem~\ref{thm:iso} lies in the construction of 
alternating bilinear maps from graphs. Given two graphs $G, 
H\subseteq\binom{[n]}{2}$, construct $\phi_G$ and $\phi_H$ as in 
Step~\ref{item:step1}. It is 
obvious that if $G$ and $H$ are isomorphic, then $\phi_G$ and $\phi_H$ are 
isomorphic. The converse direction is the elusive one, and showing its 
correctness is the main 
technical result in this note. 
%
%As alternating bilinear maps serve as the intermediate object in the 
%Baer-Lov\'asz-Tutte construction, it is natural to prove Theorem~\ref{thm:iso} by 
%first 
%
%Let $G, H\subseteq\binom{[n]}{2}$, and let $\bA$ and $\bB$ be ordered bases of 
%$\cA_G$ and $\cA_H$, respectively. It is clear that if $G$ and $H$ are 
%isomorphic, 
%then $\phi_{\bA}$ and $\phi_{\bB}$ are isomorphic, regardless of the choices of 
%ordered bases of $\cA_G$ and $\cA_H$. This allows us to write $\phi_G$ as an 
%alternating bilinear map to represent any $\phi_{\bA}$ where $\bA$ is an ordered 
%basis of $\cA_G$. 
%
%Let $\phi, \psi:\F_p^n\times\F_p^n\to\F_p^m$ be two alternating bilinear maps. 
%Then it is well-known that $\phi$ and $\psi$ are isomorphic if and only if 
%$P_\phi$ and $P_\psi$ are isomorphic. Indeed, Baer's construction 
%
%It follows that starting from two graphs $G$ 
%and $H$, if $G$ and $H$ are isomorphic as graphs, then $P_{\phi_G}$ and 
%$P_{\phi_H}$ are isomorphic as groups. However, it is not a priori clear whether 
%the inverse direction holds. Our first result shows that this is indeed the case. 
\begin{proposition}\label{prop:main}
Let $G, H\subseteq \binom{[n]}{2}$ be two graphs, and let $\phi_G, 
\phi_H:\F^n\times\F^n\to\F^m$ be two alternating bilinear maps constructed from 
$G$ and $H$ in Step~\ref{item:step1}. Then $G$ and $H$ are isomorphic if and only 
if $\phi_G$ and 
$\phi_H$ are isomorphic. 
\end{proposition}
%Note that Proposition~\ref{prop:main} holds over any field. 

\paragraph{A notion of graph homomorphisms.} 
As pointed out by J. B. Wilson in \cite[Sec. 3.3]{Wil09a}, Baer's construction 
naturally leads to a functor from the category of alternating bilinear maps over 
$\F_p$ to 
the category of $p$-groups of class $2$ and exponent $p$. To define a category of 
alternating bilinear maps, let us review 
the notion of homomorphisms for alternating bilinear maps. Let 
$\phi:\F_p^n\times\F_p^n\to\F_p^m$ and $\psi:\F_p^{n'}\times\F_p^{n'}\to\F_p^{m'}$ 
be two alternating bilinear maps. Then a \emph{homomorphism} from $\phi$ to $\psi$ 
is $(C, D)\in \M(n'\times n, \F)\times\M(m'\times m, \F)$, such that for any $u, 
v\in \F_p^n$, we have $D(\phi(u,v))=\psi(C(u), C(v))$. 
%Endowed with this 
%homomorphism notion, the construction in Step~\ref{item:step2} naturally leads to 
%a functor 

%This homomorphism notion 
%leads to a category of alternating bilinear maps over $\F_p$ which is 
%functorially 
%equivalent to the category of $p$-groups of class $2$ and exponent $p$. 

%It is then natural to wonder about if there is a homomorphism notion of graphs 
%which would make a category of graphs, such that the Baer-Lov\'asz-Tutte 
%construction is a functor from this category of graphs to the category of groups 
%through the category of alternating bilinear maps. 

Our next goal is to establish a category of graphs which relates to the category 
of groups via the category of alternating bilinear maps. This leads us to define 
the 
following notion of graph homomorphisms.
\begin{definition}\label{def:main}
Let $G\subseteq\binom{[n]}{2}$ and $H\subseteq\binom{[n']}{2}$ be two graphs. An 
injective partial function $f:[n]\to[n']$ is a \emph{pullback homomorphism} from 
$G$ to 
$H$, if 
$\forall \{i,j\}\in \binom{[n]}{2}$, $\{f(i), f(j)\}\in H\Rightarrow \{i,j\}\in 
G$. 
\end{definition} 
Note that $\{f(i), f(j)\}\in H$ implies that $f(i)$ and $f(j)$ are both defined in 
$f$. 

The above definition has the following graph-theoretic interpretation. 
For a partial function $f:X\to Y$, we call $X$ the domain of $f$ and $Y$ the 
codomain of $f$. We use $\im(f):=f(X)\subseteq Y$ to denote the image of $f$, and
$\D(f)=\{x\in X : f(x)\text{ is defined}\}$ to denote the domain of definition of 
$f$. Then a homomorphism $f$ from $G$ to $H$ just says that an isomorphic copy of 
the induced subgraph $H[\im(f)]$ is contained in $G[\D(f)]$ as a subgraph. From 
this, it is easy to verify that a composition of two homomorphisms is again a 
homomorphism; see Observation~\ref{fact:hom}.

This pullback homomorphism notion differs from the well-established notion of 
homomorphisms of graphs (see e.g. \cite{HN04}), which defines a homomorphism from 
$G$ to $H$ as a (total) function $g:[n]\to [n']$ satisfying $\{i,j\}\in 
G\Rightarrow\{f(i), f(j)\}\in H$. We shall call such homomorphisms  
\emph{pushforward homomorphisms}.
% to distinguish these two notions of homomorphism. 

The 
alert reader may note that, to naturally arrive at pullback homomorphisms, we 
should 
view $\{i',j'\}\in \binom{[n']}{2}$ as some function-like object. Indeed, one can 
identify $S=\{i',j'\}$ with its indicator function $I_S:[n']\to\{0, 1\}$, by 
$I_S(k')=1$ if and only if $k'=i'$ or $k'=j'$. Then
$f:[n]\to[n']$ naturally pulls $I_S$ back to $I_S\circ 
f:[n]\to\{0,1\}$. So alternatively, we can also define a partial injective 
function $f:[n]\to 
[n']$ to be a pullback homomorphism if it sends $\{I_S : S\in H, S\subseteq 
\im(f)\}$ to a subset of $\{I_T : T\in G\}$. This provides one explanation for 
imposing the injectivity condition, which ensures that if $S\subseteq \im(f)$, 
then 
$I_S\circ f=I_T$ for some $T\subseteq \binom{[n]}{2}$. 
%the pushforward homomorphism notion comes naturally when we view $\{i, 
%j\}\in \binom{[n]}{2}$ as a subset, while the pullback homomorphism notion comes 
%naturally when we view $\{i,j\}$ as a 

We shall examine the notion of pullback homomorphisms more closely in 
Section~\ref{sec:hom}. For now we just provide some brief 
remarks. 
First, with pullback homomorphisms and some appropriate
measure to take care of graphs with vertex 
sets not necessarily $[n]$, we can define a category of graphs, and the 
Baer-Lov\'asz-Tutte construction naturally leads to a functor from this category 
of graphs 
to the category of groups. 
%Second, it is possible to use any partial function 
%instead of injective ones in Definition~\ref{def:main}. We still think that 
%imposing injectivity is more natural. For example, using notations from 
%the last paragraphs, 
Second, several classical algorithmic problems for graphs can be formulated 
naturally using pullback homomorphisms.
These will be discussed in more details in Section~\ref{sec:hom}. We leave a more 
thorough study into this notion to future works. 
%
%It follows that the composition of two 
%homomorphisms is 
%also a homomorpism, . Third, with this homomorphism notion, we have a category of 
%graphs, and the Baer-Lov\'asz-Tutte procedure becomes a functor from this 
%category 
%of graphs to the category of groups. 

\subsection{Related works}\label{subsec:related}

\paragraph{Works related to the Baer-Lov\'asz-Tutte construction.} Let us explain 
the contexts of some works related to this construction. 
%provide some 
%context for this construction. 

Tutte used the following linear algebraic 
construction in his study of perfect matchings on graphs \cite{Tut47}. That is, 
given 
$G\subseteq \binom{[n]}{2}$, Tutte constructed an alternating matrix 
$A=(a_{i,j})$, such that for $1\leq i<j\leq n$, $a_{i,j}=x_{i,j}$ if $\{i,j\}\in 
G$, and $a_{i,j}=0$ otherwise, where $x_{i,j}$'s are independent variables. (Then 
$a_{i,j}$ for $i=j$ and $i>j$ are set by the 
alternating condition.) This alternating matrix can be easily interpreted 
as a 
linear space of alternating matrices, spanned by elementary alternating matrices 
$A_{i,j}$, $\{i,j\}\in G$. Lov\'asz explored Tutte's construction from the 
perspective of  
randomised algorithms \cite{Lov79}, and studied 
linear spaces of matrices in other topics in combinatorics \cite{Lov89}. 

Baer studied central extensions of abelian groups by abelian groups, and presented 
the construction of $p$-groups of class $2$ and exponent $p$ from alternating 
bilinear maps in 
\cite{Bae38}. Given a $p$-groups of class $2$ and exponent $p$, one can take the 
commutator bracket to obtain an alternating bilinear map. Taking commutators and 
Baer's construction form the backbone of a pair of functors between alternating 
bilinear maps over $\F_p$ and $p$-groups of class $2$ and exponent $p$, as shown 
in \cite{Wil09a}.

\paragraph{Alternating matrix spaces and the related works.} In 
Step~\ref{item:step1}, we defined $\bA_G\in\Lambda(n, \F)^m$ from a graph $G$. 
Following Lov\'asz \cite{Lov79}, set $\cA_G$ to be the linear span of matrices in 
$\bA_G$, which is a subspace of $\Lambda(n, \F)$. The alternating matrix space 
$\cA_G$ is closely related to the alternating bilinear map 
$\phi_G:\F^n\times\F^n\to\F^m$, and this perspective is easier to work with when 
making
connections to graphs. 

For example, Tutte and Lov\'asz noted 
that the matching number of $G$ 
%$\mu(G)$, 
is equal 
to one half of the maximum rank over matrices in $\cA_G$.\footnote{This is 
straightforward to 
see if the underlying field $\F$ is large enough. If $\F$ is small, it follows 
e.g. as a consequence of the linear matroid parity theorem; cf. the discussion 
after \cite[Theorem 4]{Lov89}.}
Recently in \cite{BCG+19}, it is shown
that the independence number of $G$
%, $\alpha(G)$, 
is 
equal to the 
maximum 
dimension over the totally isotropic spaces\footnote{A subspace $U\leq 
\F^n$ is a totally
isotropic space of $\cA\leq \Lambda(n, \F)$, if for any $u, u'\in U$, and any 
$A\in \cA$, $u^tAu'=0$.} of $\cA_G$. They also showed %\yinan{Used to be: show} 
that the 
chromatic number of $G$ %, $\chi(G)$, 
is equal to the minimum $c$ such that there 
exists a direct sum decomposition of $\F^n$ into $c$ non-trivial totally isotropic 
spaces 
for $\cA_G$. In 
\cite{LQ19}, the 
vertex and edge connectivities of a graph $G$ are 
shown to be equal to certain parameters related to orthogonal 
decompositions\footnote{A direct sum decomposition $\F^n=U\oplus V$ is orthogonal 
with 
respect to $\cA\leq\Lambda(n, \F)$, if for any $u\in U, v\in V$, and $A\in\cA$, 
$u^tAv=0$. 
This definition is probably more natural from the alternating bilinear 
map viewpoint, but the two parameters in \cite{LQ19} are arguably easier to 
motivate from the alternating matrix space perspective.
} of 
$\cA_G$.

Given two alternating matrix spaces $\cA, \cB\leq \Lambda(n, \F)$, $\cA$ and $\cB$ 
are isomorphic, if there exists $T\in\GL(n, \F)$, such that $\cA=T^t\cB 
T:=\{T^tBT: B\in\cB\}$. Theorem~\ref{thm:iso} then sets up another connection 
between graphs and 
alternating matrix spaces in the context of isomorphisms. In \cite{LQ17}, 
alternating matrix space isomorphism problem was 
studied as a linear algebraic analogue of graph isomorphism problem. The algorithm 
in \cite{LQ17} was recently improved in \cite{BGL+19}.

\paragraph{An implication: simplifying a reduction from graph isomorphism to group 
isomorphism.} To test whether two 
graphs are isomorphic is a celebrated algorithmic problem in computer science. 
Babai's recent breakthrough showed that the graph isomorphism problem can be 
solved in quasipolynomial time \cite{Bab16}. Babai proposed the group isomorphism 
problem as one of the next targets to study for isomorphism testing 
problems. 

The relations between graph isomorphism problem and group isomorphism problem are 
as 
follows. On the one hand, when groups are given by Cayley tables, group 
isomorphism 
reduces to 
graph isomorphism \cite{KST93}. On the other hand, graph isomorphism reduces to 
group 
isomorphism, when groups are given by generators as 
permutations \cite{Luks} or matrices over finite fields \cite{GQ19}. The reduction 
in \cite{GQ19} actually constructs $p$-groups of class $2$ and exponent $p$ from 
graphs.

The reduction in \cite{GQ19} relies a reduction from graph isomorphism to 
alternating bilinear map isomorphism. However, there some 
gadget is needed to restrict from invertible matrices to monomial matrices. 
Theorem~\ref{thm:iso} then 
implies that that gadget is actually not needed, considerably simplifying the 
reduction.

\paragraph{Structure of this note.} In the following, we shall first prove 
Proposition~\ref{prop:main} in Section~\ref{sec:iso}. We then discuss 
Definition~\ref{def:main} in Section~\ref{sec:hom}.

\section{Proof of Proposition~\ref{prop:main}}\label{sec:iso}

\paragraph{Further notations.} %Let $\F$ be a field. 
%Let $\M(n, \F)$ be the linear 
%space of $n\times n$ matrices 
%over $\F$. 
The symmetric group on a set $V$ is denoted by $\sym(V)$. For $n\in\N$, we 
let $\S_n=\sym([n])$. For $B\in \M(\ell\times n, \F)$, we use $B(i,j)$ to denote 
the $(i,j)$th entry of $B$. We use $e_i$ to denote the $i$th standard 
basis vector of $\F^n$. For a vector $w\in \F^n$, we use $w[i]$ to denote the 
$i$th entry of $w$. For a graph $H\subseteq\binom{[n]}{2}$, 
$\overline{H}=\binom{[n]}{2}\setminus H$ is the complement graph of $H$.

\paragraph{A reformulation of the problem.} 
Let $G, H\subseteq\binom{[n]}{2}$ be 
two graphs of size $m$. Recall that by Step~\ref{item:step1}, from $G$ and $H$ we  
construct $\bA_G$ and $\bA_H$ in $\Lambda(n, \F)^m$. Let $\cA_G$ and $\cA_H$ 
be subspaces of $\Lambda(n, \F)$ spanned by $\bA_G$ and $\bA_H$, respectively. 
Recall that $\cA_G$ and $\cA_H$ are isomorphic, if there exists $T\in\GL(n, \F)$, 
such 
that $\cA_G=T^t\cA_HT$. Clearly, $\cA_G$ and $\cA_H$ are isomorphic if and only if 
$\phi_G$ and $\phi_H$ are isomorphic. So we focus on $\cA_G$ and $\cA_H$ in the 
following. 

%We then need to show that if $\cA_G$ and $\cA_H$ are isomorphic, then $G$ and $H$ 
%are isomorphic. Suppose that $T\in \GL(n, \F)$ satisfies $\cA_G=T^t\cA_HT$.
%Let us examine what restrictions we have on $T$. Let the $i$th row of $T$ be 
%$w_i^t$, where $w_i\in\F^n$. 
%Suppose $\{k, l\}\in 
%H$. Then $T^t (e_k e_l^t-e_le_k^t) T= w_kw_l^t-w_lw_k^t$, which belongs to 
%$\cA_G$. Note that a matrix $B\in 
%\cA_G$ if and only if $\forall \{i, j\}\not\in G$, $B(i, j)=0$. Therefore, for 
%any 
%$\{k, l\}\in H$ and $\{i, j\}\not\in G$, $(w_kw_l^t-w_lw_k^t)(i, 
%j)=w_k(i)w_l(j)-w_l(i)w_k(j)=0$. 
%The above reasoning can be 
%reversed 
%to obtain 
We then need to show that if $\cA_G$ and $\cA_H$ are isomorphic, then $G$ and $H$ 
are isomorphic. Suppose that $T\in \GL(n, \F)$ satisfies $\cA_H=T^t\cA_GT$.
Let us examine what restrictions we have on $T$. Let the $i$th row of $T$ be 
$w_i^t$, where $w_i\in\F^n$. 
Suppose $\{i, j\}\in 
G$. Then $T^t (e_i e_j^t-e_je_i^t) T= w_iw_j^t-w_jw_i^t$, which belongs to 
$\cA_H$. Note that a matrix $B\in 
\cA_H$ if and only if $\forall \{k,l\}\not\in H$, $B(k,l)=0$. Therefore, for any 
$\{i,j\}\in G$ and $\{k,l\}\not\in H$, 
$(w_iw_j^t-w_jw_i^t)(k,l)=w_i[k]w_j[l]-w_j[k]w_i[l]=0$. 
This leads to the following observation. 
%
%Suppose $G=\{\{s_1, t_1\}, \dots, 
%\{s_m, t_m\}\}$, and $H=\{\{u_1, v_1\}, \dots, \{u_m, v_m\}\}$, where $s_i, t_j, 
%u_k, v_\ell\in[n]$. 
%
%Let $e_i$ be the $i$th standard basis vector of $\F^n$. Then $\cA_G$ is spanned 
%by 
%$e_{s_i}e_{t_i}^t-e_{t_i}e_{s_i}^t$, and $\cA_H$ is spanned by 
%$e_{u_i}e_{v_i}^t-e_{v_i}e_{u_i}^t$. Let the $i$th row of $T$ be $w_i^t$. Then 
%$T^t (e_{u_i}e_{v_i}^t-e_{v_i}e_{u_i}^t) T= w_{u_i}w_{v_i}^t-w_{v_i}w_{u_i}^t$. 
%Note that 
%$(w_{u_i}w_{v_i}^t-w_{v_i}w_{u_i}^t)(a, 
%b)=w_{u_i}(a)w_{v_i}(b)-w_{v_i}(a)w_{u_i}(b).$ 
%
%Observe that a matrix $B\in 
%\cA_G$, if and only if for any $\{a, b\}\not\in G$, $B(a, b)=0$. Since
%$w_{u_i}w_{v_i}^t-w_{v_i}w_{u_i}^t\in \cA_G$, it follows that 
%$w_{u_i}(a)w_{v_i}(b)-w_{v_i}(a)w_{u_i}(b)=0$. The above reasoning can be 
%reversed 
%to obtain this following observation. 
\begin{observation}\label{obs:main}
Let $\cA_G$ and $\cA_H$ be alternating matrix spaces in $\Lambda(n, \F)$ 
constructed from two graphs $G$ and $H$, respectively. Then $\cA_G$ and $\cA_H$ 
are isomorphic, if and only if there exists $T\in \GL(n, \F)$, such that for any 
$\{i,j\}\in G$ and $\{k, l\}\not\in H$, the $2\times 2$ submatrix of $T$ with row 
indices $\{i, j\}$ and column indices $\{k, l\}$ has the zero determinant. 
\end{observation}

We then show that the existence of such $T$ 
would imply that the graphs $G$ and $H$ are isomorphic.

\begin{proposition}\label{prop:key}
Let $G, H\subseteq\binom{[n]}{2}$ be two graphs. Suppose there exists invertible 
matrix $T\in\GL(n, \F)$, such that for any 
$\{i,j\}\in G$ and $\{k,l\}\notin H$, 
$\det
\begin{bmatrix}
T(i,k) & T(i,l)\\
T(j,k) & T(j,l)
\end{bmatrix}=0$.
Then $G$ and $H$ are isomorphic.
\end{proposition}

By Observation~\ref{obs:main}, Proposition~\ref{prop:key} implies 
Proposition~\ref{prop:main}. We then focus on proving Proposition~\ref{prop:key} 
in 
the following.

\subsection{Proof of Proposition~\ref{prop:key}}

%Let us first  prepare for the proof of Theorem~\ref{thm:main}.
Given $\sigma\in \S_n$ and $M\in \M(n, \F)$, let $f(M,\sigma):= \text{sgn} 
(\sigma)\prod_{i=1}^nM(i,\sigma(i))$.
%, for a $n\times 
%n$ matrix $M$ and $n$-order permutation 
%$\sigma=(\sigma(1),\sigma(2),\dots,\sigma(n))$.\\
For $S\subseteq[n]$, $\sigma(S):=\{\sigma(x) : x\in S\}$. 
Given $A\subseteq[n]$ and $B\subseteq [n]$,
let $M(A,B)$ be the submatrix of $M$ with row indices in $A$ and column indices 
in $B$. 
%$\sigma(S):=\{\sigma(x)|x\in S\}$, for a set $S\subseteq\{1,2,\dots,n\}$ and 
%$n$-order permutation $\sigma$.\\

To prepare for the proof, we need the following definition. For $M,N\in\M(n, \F)$, 
define $M\preceq N$ if  $M(i,j)=0$ or $M(i,j)= N(i,j)$ for 
all $(i,j)$. In other words, $M\preceq N$ if the non-zero entries of $M$ are equal 
to the corresponding ones in $N$. We also use $M\prec N$ to denote that $M\neq N$ 
and $M\preceq N$. 

\begin{definition}\label{def:sem}
Let $G, H\subseteq \binom{[n]}{2}$ be two graphs. Let 
$T\in\GL(n, \F)$ such 
that for any 
$\{i,j\}\in G$ and $\{k,l\}\notin H$, 
$\det
\begin{bmatrix}
T(i,k) & T(i,l)\\
T(j,k) & T(j,l)
\end{bmatrix}=0$.
%
%Given two $n$-vertexes undirected graphs $G$ and $H$ and a $n\times n$ matrix 
%$T\in\mathbb{F}^{n\times n}$ that $
%\begin{vmatrix}
%T(i,k) & T(i,l)\\
%T(j,k) & T(j,l)
%\end{vmatrix}=0
%$ holds for all $(i,j)\in E_G$ and $(k,l)\notin E_H$. 
We define a subset of matrices, $\bK(G, H, T)\subseteq \M(n, \F)$, such that $M\in 
\bK(G, H, T)$, 
if there exists a 
permutation $\sigma\in \S_n$, and a disjoint partition of 
$[n]=\bigcup_{k=1}^r S_k$ such that:
\begin{enumerate}
\item $r<n$; 
\item every $S_k$ is connected in $G$; 
\item every $\sigma(S_k)$ is connected in $\overline{H}$;
\item $M\preceq T$; 
\item $M_{ij}\neq 0$ if and only if there exists $k\in[r]$, such that $i\in S_k$ 
and $j\in \sigma(S_k)$;
\item $\forall k\in[r]$, the rank of the submatrix $M(S_k,\sigma(S_k))$ is $1$. By 
(4), this implies that $T(S_k, \sigma(S_k))$ is of rank $1$.
\end{enumerate}
\end{definition}

Note that for any $M\in \bK(G, H, T)$, we have $\det(M)=0$. This is because
$$\det(M)=\pm\prod_{k=1}^r\det(M(S_k,\sigma(S_k))), $$
and since $r<n$, there exists $k\in[r]$ 
such that $|S_k|\ge 2$ with $\rk(M(S_k,\sigma(S_k)))=1$. 
%
%We are now ready to prove Proposition~\ref{thm:main}.
%
%\begin{proof}[Proof of Theorem~\ref{thm:main}]

We then state the following two lemmas, whose proofs will be postponed to the 
end of this subsection.

\begin{lemma}\label{lem:divide}
Given a matrix $T\in \M(n, \F)$, if there exists a set of matrices $\{T_1, 
T_2,\dots,T_s\}\subseteq \M(n, \F)$ such that 
\begin{enumerate}
\item $T_i\preceq T$ for all $i$, %\ynote{Is this condition necessary?},
\item for any $\sigma\in \S_n$, if $f(T,\sigma)\neq 0$ then there exists $i\in[s]$ 
such 
that $f(T_i,\sigma)=f(T,\sigma)$,
\item for any $\sigma\in \S_n$ and $i\neq j$, $f(T_i,\sigma)f(T_j,\sigma)=0$,
\end{enumerate}
then $\det(T)=\sum_{i=1}^s \det(T_i)$.
\end{lemma}

We defined a set of matrices $\bK(G, H, T)$ in Definition~\ref{def:sem}. Note that 
$\bK(G, H, T)$ is a finite set, so there are maximal elements in it with respect 
to the  $\preceq$ order. We now 
show that the set of all maximal matrices in $\bK(G, 
H, T)$ satisfies the conditions of 
Lemma~\ref{lem:divide}.
\begin{lemma}\label{lem:maximal}
If $G$ is not isomorphic to $H$, then the set of all maximal matrices in $\bK(G, 
H, 
T)$ satisfies the conditions of Lemma~\ref{lem:divide}.
\end{lemma}

Given these two lemmas, we can prove Proposition~\ref{prop:key} by way of 
contradiction. That is, 
%We shall prove the inverse negative propositions of Theorem~\ref{thm:main} by  
%contradiction. 
suppose $H$ is not isomorphic to $G$, but there exists $T\in\GL(n, 
\F)$, such that for any $(i,j)\in 
G$ and $(k,l)\notin H$, $\det
\begin{bmatrix}
T(i,k) & T(i,l)\\
T(j,k) & T(j,l)
\end{bmatrix}=0$. Let $\{T_1,T_2,\dots,T_m\}$ be the set of all maximal matrices 
in $\bK(G, H, T)$, then $\det(T)=\sum_{i=1}^m\det(T_i)$ by Lemma~\ref{lem:divide} 
and 
Lemma~\ref{lem:maximal}. Note that the determinant of any matrix in $\bK(G, H, 
T)$  
is 
$0$, so $\det(T)=0$. This contradicts with the fact that $T$ is invertible.

We now prove the two key lemmas. 
\begin{proof}[Proof of Lemma~\ref{lem:divide}]
The statement can be verified as follows.
\begin{align*}
\sum_{i=1}^s\det(T_i)=&\sum_{i=1}^s \sum_{\sigma\in \S_n} f(T_i,\sigma)\\
=&\sum_{\sigma\in \S_n} \sum_{i=1}^s f(T_i,\sigma)\\
=&\sum_{\sigma\in \S_n:\exists i,f(T_i,\sigma)\neq 0}\sum_{i=1}^sf(T_i,\sigma).
\end{align*}
Then by condition 3, for any $\sigma\in \S_n$ with $\exists 
i\in [s],f(T_i,\sigma)\neq 0$, there is actually a unique such $i$. We denote it 
by 
$i_\sigma$. Let $P=\{\sigma\in \S_n : \exists i\in[s], f(T_i, \sigma)\neq 0\}$. 
Let $Q=\{\sigma\in \S_n : f(T, \sigma)\neq 0\}$. By condition 2, we have 
$Q\subseteq P$. By condition 1, we have $P\subseteq Q$. So we can continue 
extending the equations from the above as 
follows. 
\begin{align*}
\sum_{i=1}^s\det(T_i)=&\sum_{\sigma\in P}f(T_{i_\sigma},\sigma)\\
=&\sum_{\sigma\in Q}f(T,\sigma) \\ %\tag{by condition 2}\\
=&\sum_{\sigma\in \S_n}f(T,\sigma)\\
=&\det(T). \qedhere
\end{align*}
\end{proof}

\begin{proof}[Proof of Lemma~\ref{lem:maximal}]
Let $\{T_1,T_2,\dots,T_m\}$ be the set of all maximal
matrices in $\bK(G, H, 
T)$. Let us verify the conditions in Lemma~\ref{lem:divide} one by one.
\begin{description}
\item[Condition 1:] $T_i\preceq T$ for all $i$. 

This follows directly from condition 4 in the definition of $\bK(G, H, T)$.
\item[Condition 2:] For any $\sigma\in \S_n$, if $f(T,\sigma)\neq 0$ then there 
exists 
$i$ that $f(T_i,\sigma)=f(T,\sigma)$.

Since
$G$ is not isomorphic to $H$, for any permutation $\sigma\in \S_n$ with 
$f(T,\sigma)\neq 0$, there exists $\{i,j\}\in\binom{[n]}{2}$ such that $\{i,j\}\in 
G$ 
and 
$\{\sigma(i),\sigma(j)\}\notin H$. This implies that 
$$\det\begin{bmatrix}
T(i,\sigma(i))&T(i, \sigma(j))\\
T(j,\sigma(i))&T(j,\sigma(j))
\end{bmatrix}=0,$$ 
that is,  
$T(i,\sigma(j))T(j,\sigma(i))=T(i,\sigma(i))T(j,\sigma(j))\neq 0$. Note that 
$T(i, \sigma(i)T(j,\sigma(j))\neq 0$, as $f(T, \sigma)\neq 0$. Construct a 
matrix $T'\in \M(n, \F)$, such that $T'$ coincide with $T$ in 
%such that $T'(k, \sigma(k))=T(k, \sigma(k))$ for $k\in[n]$, $T'(i, 
%\sigma(j))=T(i, \sigma(j)$, $T'(j, \sigma(i))=T(j, \sigma(i))$, and $T'(k, 
%\ell)=0$ 
these $n+2$ entries 
$$\{(1,\sigma(1)),\dots,(n,\sigma(n)),(i,\sigma(j)),(j,\sigma(i))\},$$ and then 
set 
other entries of $T'$ to be $0$. Then $T'\in \bK(G, H, T)$, and 
$f(T',\sigma)=f(T, \sigma)$. It follows that there exists $T_i$ 
such that $T_i$ is maximal in $\bK(G, H, T)$,
$T_i\succeq T'$, and it satisfies that $f(T_i,\sigma)=f(T,\sigma)$.

\item[Condition 3:] For any $\sigma\in \S_n$ and $i\neq j$, 
$f(T_i,\sigma)f(T_j,\sigma)=0$.

For the sake of contradiction, suppose there exist $\sigma\in \S_n$, and $i\neq 
j$, such that $f(T_i,\sigma)\neq 
0$ and $f(T_j,\sigma)\neq 0$. Since $T_i$ and $T_j$ satisfy 
Definition~\ref{def:sem}, there are partitions of $[n]$ associated with them, and 
we denote the partition of $[n]$ associated with $T_i$ (resp. $T_j$) by $U^{(i)}$ 
(resp. $U^{(j)}$). In the following we show how to construct $P\in \bK(G,H,T)$ 
such 
that $T_i\prec P$, arriving at the desired contradiction. 

Let 
$U$ be a set of subsets of $[n]$, obtained by taking the union of those subsets 
in $U^{(i)}$ and 
$U^{(j)}$. %, which is a set of subsets of $[n]$. 
Recall that for any $S\in U$, 
$S$ is connected in $G$, 
$\sigma(S)$ is 
connected in $\overline{H}$,  all elements in submatrix $T(S,\sigma(S))$ 
are 
nonzero, and the rank of $T(S,\sigma(S))$ is $1$. 

We now implement the following procedure, which transforms $U$ into a partition of 
$[n]$.
\begin{itemize}
\item As long as there exit $S_1, S_2\in U$ such that $S_1\cap 
S_2\neq\emptyset$, we perform one of the following: 
\begin{itemize}
\item If $S_1\subseteq S_2$, delete $S_1$ from $U$.
\item Else if $S_1\supseteq S_2$, delete $S_2$ from $U$.
\item Else delete $S_1$ and $S_2$ from $U$, and add  $S_1\cup S_2$ 
to $U$. 
\end{itemize}
\end{itemize}
%Repeat above procedure until elements in $S$ are disjoint. 
The key is that after each step, we can maintain the property that for any $S\in 
U$, $S$ is connected in $G$, 
$\sigma(S)$ is 
connected in $\overline{H}$,  all elements in the submatrix $T(S,\sigma(S))$ 
are 
nonzero, and the rank of $T(S,\sigma(S))$ is $1$. When $S_1\subseteq S_2$ or 
$S_1\supseteq S_2$, this is clear. We shall prove that this also holds in the last 
case in Claim~\ref{lem:unisqu}.
%
%When 
%By Lemma~\ref{lem:unisqu},  $S_{k_1}\cup S_{k_2}$ is connected in $G$, 
%$\sigma(S_{k_1}\cup S_{k_2})$ is connected in $\bar{H}$, and all elements in 
%$T(S_{k_1}\cup S_{k_2},\sigma(S_{k_1}\cup S_{k_2}))$ are nonzero and rank of 
%$T(S_{k_1}\cup S_{k_2},\sigma(S_{k_1}\cup S_{k_2}))$ is $1$. 

Therefore, after performing the above procedure, $U$ becomes a partition 
of $\{1,2,\dots,n\}$, and for any $S\in U$, all elements in the submatrix 
$T(S,\sigma(S))$ are nonzero, and the rank of $T(S,\sigma(S))$ is $1$. It 
follows that there exists $P\in \bK(G, H, T)$ corresponding to the partition $U$  
and $\sigma\in \S_n$, 
by 
Definition~\ref{def:sem}. It is clear that $T_i\prec P$, giving us the desired 
contradiction. \qedhere
\end{description}
\end{proof}

\begin{claim}\label{lem:unisqu}
Suppose we have two graphs $G, H\subseteq\binom{[n]}{2}$, $T\in \GL(n, \F)$, 
$\sigma\in \S_n$, and $S_1,S_2\subseteq[n]$, which satisfy the 
following conditions:
\begin{itemize}
\item[(a)] $\forall (i,j)\in G$ and $\forall (k,l)\in \overline{H}$, 
$\det
\begin{bmatrix}
T(i,k) & T(i,l)\\
T(j,k) & T(j,l)
\end{bmatrix}=0
$,
\item[(b)] $S_1\cap S_2\neq\emptyset$,
\item[(c)] $S_1$ and $S_2$ are both connected in $G$,
\item[(d)] $\sigma(S_1)$ and $\sigma(S_2)$ are both connected in $\overline{H}$,
\item[(e)] elements in the submatrix $T(S_1,\sigma(S_1))$ are all nonzero, and 
$\rk(T(S_1,\sigma(S_1)))=1$,
\item[(f)] elements in the submatrix $T(S_2,\sigma(S_2))$ are all nonzero,
and $\rk(T(S_2,\sigma(S_2)))=1$.
\end{itemize}
Then $S_1\cup S_2$ is connected in $G$, $\sigma(S_1\cup S_2)$ is connected in 
$\overline{H}$, all elements in the submatrix $T(S_1\cup S_2,\sigma(S_1\cup S_2))$ 
are 
nonzero, and $\rk(T(S_1\cup S_2,\sigma(S_1\cup S_2)))=1$.
\end{claim}
\begin{proof}
The connectivities of $S_1\cup S_2$ in $G$ and 
$\sigma(S_1\cup S_2)$ in $\overline{H}$ are evident by (c) and (d). We then focus 
on the last two 
properties regarding $T(S_1\cup S_2,\sigma(S_1\cup S_2))$ in the statement.

As $S_1\cap S_2\neq\emptyset$ by (b), there exists some $u\in S_1\cap S_2$. Fix 
one such $u\in S_1\cap S_2$. Our 
goal is to prove that for any $v\in S_1\cup S_2$, 
and $w\in \sigma(S_1\cup S_2)$, $T(v, w)=\lambda_v T(u, w)$ holds, for some 
$\lambda_v\neq 0\in \mathbb{F}$ that only depends on $v$. We shall 
achieve this by 
applying double inductions on 
the 
distance between $v$ and $u$ (the outer induction), and on the distance between 
$w$ and $\sigma(u)$ (the inner induction). 
%respectively to prove there exists constant 
%$\lambda_v\in \mathbb{F}$ that $T(v, w)=\lambda_v T(u, w)$ for all $v\in S_1\cup 
%S_2$ and $w\in \sigma(S_1\cup S_2)$. 
By conditions (e) and (f), and since $u\in S_1\cap S_2$, we have 
$T(v,\sigma(u))\neq 
0$ for any $v\in S_1\cup S_2$. So for any $v\in S_1\cup S_2$, we can set 
\begin{equation}\label{eq:lambda}
\lambda_v=T(v,\sigma(u))/ 
T(u,\sigma(u)).
\end{equation}
%; note that $u$ is a fixed element in $S_1\cap S_2$. 
Let $d_G(u,v)$ 
denote the 
distance from $u$ to $v$ in graph $G$.
\begin{description}
\item[The initial condition for the outer induction.] By condition (e) and (f), 
for any $w\in \sigma(S_1\cup S_2)$ we have 
$T(u, w)\neq 0$. So when $d_G(u,v)=0$, i.e. $u=v$, it trivially holds that for any 
$w\in \sigma(S_1\cup S_2)$, $T(u, w)=\lambda_v 
T(u, w)\neq 0$ where $\lambda_v=1$. 
\item[The outer inductive step.] Suppose $T(v, w)=\lambda_v T(u, w)\neq 0$ holds 
for 
any $v\in S_1\cup 
S_2$ satisfying that $d_G(u,v)\leq t_1$, and any $w\in \sigma(S_1\cup S_2)$. 
%Assuming when $d_G(u,v)\le t_1$, $T(v, w)=\lambda_v T(u, w)\neq 0$ holds for such 
%$v\in S_1\cup S_2$ and $w\in\sigma(S_1\cup S_2)$. 
Consider some vertex $j\in S_1\cup S_2$ with $d_G(u,j)=t_1+1$.
%for $j\in 
%S_1\cup S_2$, then 
Then there exists some other vertex $i\in S_1\cup S_2$ such that $\{i,j\}\in G$ 
and 
$d_G(u,i)= 
t_1$. We now prove that $T(j, w)=\lambda_j T(u, w)$ for any $w\in \sigma(S_1\cup 
S_2)$, by induction on $d_{\overline H}(\sigma(u), w)$.
\begin{description}
\item[The initial condition for the inner induction.] As the base case, suppose 
$d_H(\sigma(u), w)=0$, i.e. $w=\sigma(u)$. By 
conditions (e) and (f), $T(j,\sigma(u))\neq 0$, so $T(j,\sigma(u))=\lambda_j 
T(u,\sigma(u))\neq 0$ by the definition of $\lambda_j$ in Equation~\ref{eq:lambda}.
%when $d_{\bar{H}}(\sigma(u),l)=0$.
\item[The inner inductive step.] Suppose we have 
$T(j,k)=\lambda_j T(u,k)\neq 0$ for any $k\in \sigma(S_1\cup S_2)$ satisfying 
$d_{\overline H}(\sigma(u), k)\leq t_2$. Consider  $l\in \sigma(S_1\cup S_2)$ 
satisfying
$d_{\overline H}(\sigma(u),l)=t_2+1$. Then there exists 
$k\in \sigma(S_1\cup S_2)$ that $\{k,l\}\in \overline H$ and $d_{\overline 
H}(\sigma(u),k)= 
t_2$.
We have $\{i,j\}\in G$, $\{k,l\}\in \overline H$, $T(i,k)=\lambda_i T(u,k)\neq 0$ 
(by 
the inductive hypothesis on $d_G(u,i)$), 
$T(i,l)=\lambda_i T(u, l)\neq 0$ (by the inductive hypothesis on $d_G(u,i)$), and 
$T(j,k)=\lambda_j T(u,k)\neq 0$ (by the inductive hypothesis on $d_H(\sigma(u), 
k)$). It follows 
from 
condition (a) that $T(j,l)=\lambda_j T(u, l)\neq 0$.
\item [Concluding the inner induction.] By the above, we have that  
$T(j, w)=\lambda_j T(u, w)$ for 
any $w\in \sigma(S_2)$. So for any $w\in \sigma(S_1\cup S_2)$, $T(j, 
w)=\lambda_{j} 
T(u, w)\neq 0$ when $d_G(u,j)= t_1+1$. 
\end{description}
\item[Concluding the outer induction.]
By the above, we have that $T(v, w)=\lambda_{v} T(u, w)\neq 0$ for any 
$v\in S_1\cup S_2$ and $w\in \sigma(S_1\cup S_2)$.\qedhere
\end{description}
\end{proof}

\section{Discussions on pullback homomorphisms}\label{sec:hom}

\paragraph{A functor based on the Baer-Lov\'asz-Tutte construction.} Let $X$ be a 
finite 
set. A (simple and undirected) graph $G=(X, E)$ consists of a vertex set $X$ and 
an edge set $E\subseteq\binom{X}{2}$. Pullback homomorphisms as in 
Definition~\ref{def:main} can be easily adapted to accommodate graphs $G=(X, E)$ 
and $H=(Y, F)$ with arbitrary vertex sets. 

%Also note that the 
%Baer-Lov\'asz-Tutte procedure can be applied for graphs $G=(V, E)$ on any vertex 
%set $V$. We then present some 
%basic observations. 

Let us also present a variant of 
the Baer-Lov\'asz-Tutte construction for graphs on arbitrary vertex sets. Note 
that a little 
twist is needed so that this construction does not require total orders on 
the vertex and edge sets. 

Let $X$ be a finite set. Given a field $\F$, we use $\F X$ to denote the 
$\F$-vector space spanned by elements in $X$.
%, and $\F^X$ to denote the 
%$\F$-vector space of (total) functions from $X$ to $\F$. 
For $v\in \F X$ and $x\in 
X$, we use $v[x]$ to denote the coefficient of $x$ in $v$. 
%For $x\in X$, the 
%indicator function $\hat x: X\to \{0, 1\}$ is defined as $\hat x(x')=1$ if and 
%only if $x'=x$. It further induces a linear function on $\F X$, also denoted as 
%$\hat x$, by $\hat x(v)=v[x]$.

Suppose we have a graph $G=(X, E)$. For any $E\subseteq\binom{X}{2}$, let $\tilde 
E=\{(x, x')\in X\times X : \{x, x'\}\in E\}$. 
%For a vector space $V$ over a field 
%$\F$, let 
%$\Lambda(V)$ be the space of alternating bilinear forms on $V$. 
For $(x,x')\in X\times 
X$, the 
elementary alternating bilinear form $A_{(x,x')}:\F X\times\F X\to\F$ is defined 
as $A_{(x,x')}(u, v)=u[x]v[x']-u[x']v[x]$. Then an alternating 
bilinear 
map 
$\phi_G:\F X\times \F X\to \F\tilde E$ can be defined as $(\phi_G(u, 
v))[(x,x')]=A_{(x,x')}(u, v)$.
%can be defined by setting the coefficient of 
%$(x,x')\in \tilde E$ to be $A_{(x,x')}(f,g)$. 
Note that 
$\dim(\im(\phi_G))=|E|$, as $A_{(x, 
x')}(u,v)=-A_{(x',x)}(u,v)$ for any $u, v\in \F X$. We then apply Baer's 
construction %, which works for any alternating bilinear maps over $\F_p$, to 
to $\phi_G:\F_p X\times \F_p X\to \im(\phi_G)$ 
to obtain a $p$-group of class $2$ and exponent $p$ $P_G$. In particular, the set 
of 
group elements of $P_G$ is $\F_p X\times\im(\phi_G)$.
%More specifically, the set 
%of group elements of $P_G$ is $\F^X\times \im(A_G)$, 

We record a basic property of pullback homomorphisms in the following fact.

\begin{fact}\label{fact:hom}
Let $G_i=(X_i, E_i)$, $i\in[3]$, be three graphs. Let $f_j:X_j\to 
X_{j+1}$, $j\in[2]$, be pullback homomorphisms from $G_j$ to $G_{j+1}$. 
Then  $f_2\circ f_1:X_1\to X_3$ is a pullback homomorphism from $G_1$ to $G_3$.
\end{fact}
%\begin{proof}
%Let $f=f_2\circ f_1:X_1\to X_3$. For $\{x, x'\}\in E_1$, if $\{f(x), 
%f(x')\}\in E_3$, then it must be the case that $x$ and $x'$ are defined in $f_1$, 
%and $f_1(x)$ and $f_1(x')$ are defined in $f_2$. Furthermore, $\{f(x), f(x')\}\in 
%E_3\iff \{f_2(f_1(x)), f_2(f_1(x'))\}\in E_3$, which, by $f_1$ and $f_2$ being 
%pullback homomorphisms, implies $\{f_1(x), f_1(x')\}\in E_2$ and then $\{x, 
%x'\}\in E_1$. This shows that $f$ is a pullback homomorphism from $G_1$ to $G_3$. 
%\end{proof}

Fact~\ref{fact:hom} suggests the following category of graphs 
$\catgraphpull$, where $\leftarrow$ denotes ``pullback.'' In $\catgraphpull$, the 
objects are graphs of the form $G=(V, E)$, the morphisms are pullback 
homomorphisms, and the identity homomorphism on $G=(V, E)$ is the identity map 
from $V$ to $V$. The identity axiom and the associativity axiom hold trivially.

Let $\catgroup$ be the category of groups. The Baer-Lov\'asz-Tutte construction 
then naturally leads to the following functor from $\catgraphpull$ 
to $\catgroup$.

\begin{definition}
Let $\BLTfunctor$ be a pair of maps from objects and morphisms in $\catgraphpull$ 
to those in $\catgroup$ as follows.
\begin{itemize}
\item For $G\in \catgraphpull$, set $\BLTfunctor(G)$ to be the $p$-group of class 
$2$ and exponent $p$ via the Baer-Lov\'asz-Tutte construction. 
\item Given two graphs %$G_1=(X_1, E_1)$ and $G_2=(X_2, E_2)$, 
$G=(X, E)$ and $H=(Y, F)$, 
let %$f:X_1\to X_2$ 
$f:X\to Y$
be a pullback homomorphism. As $f$ is injective, $f$ induces an injective 
partial function 
$\tilde f:\tilde E\to \tilde F$, %\tilde E_1\to \tilde E_2$, 
in which $(x, x')\in \tilde E$ %\tilde E_1$ 
is defined in 
$\tilde f$ if and only if both $x$ and $x'$ are defined in $\tilde f$. 
Furthermore, 
$f$ induces a linear map $\ell: \F_p X\to \F_p Y$ %\F_p X_1\to \F_p X_2$ 
as follows: for 
%$u\in \F_p X_1$ 
$u\in \F_p X$
and $y\in Y$, %$x\in X_2$, 
if $y\not\in\im(f)$, then $(\ell(u))[y]=0$. Otherwise, 
$(\ell(u))[y]=u[f^{-1}(y)]$.
%$\phi$ induces $\tilde 
%\phi:\F^{X_2}\to \F^{X_1}$ as follows: for $f:X_2\to \F$ and $x_1\in X_1$, if 
%$x_1$ is not defined in $\phi$, then $(\tilde\phi(f))(x_1)=0$. Otherwise, 
%$(\tilde\phi(f))(x_1)=f(\phi(x_1))$. 
Analogously, $\tilde f$ induces a linear map $\tilde \ell:
%\F_p\tilde E_1 \to \F_p\tilde E_2$ 
\F_p\tilde E\to \F_p\tilde F$ 
as follows: for $v\in \F_p\tilde E$ and %E_1$ and 
$e\in \tilde F$, %E_2$, 
if $e\not\in \im(\tilde f)$, then 
$(\tilde \ell(v))[e]=0$. Otherwise, $(\tilde\ell(v))[e]=v[\tilde f^{-1}(e)]$. 
Clearly, $\tilde\ell$ sends $\im(\phi_{G})$ to $\im(\phi_{H})$. We then set 
$\BLTfunctor(f):\BLTfunctor(G)\to\BLTfunctor(H)$ as sending $(u, v)\in 
\F_p X\times \im(\phi_{G})$ to $(\ell(u), \tilde\ell(v))\in \F_p Y\times 
\im(\phi_{H})$.
\end{itemize}
\end{definition}

We first verify that $\BLTfunctor(f)$ is a group homomorphism. To see this, 
by \cite[Sec. 3.4]{Wil09a}, it is enough to verify that $(\ell, 
\tilde\ell)$ is an alternating bilinear map homomorphism from $\phi_{G}$ to 
$\phi_{H}$. To see this, note that for any $u, u'\in \F X$ and $(y, y')\in 
Y\times Y$, we have the following.

\begin{itemize}
\item Suppose both $y, y'$ are in $\im(f)$. On the one hand, we have 
$$\big(\tilde\ell(\phi_{G}(u, u'))\big)[(y, 
y')]=u[f^{-1}(y)]\cdot u'[f^{-1}(y')]-u[f^{-1}(y')]\cdot u[f^{-1}(y)].$$
On the 
other hand,
\begin{eqnarray*}
(\phi_{H}(\ell(u), \ell(u')))[(y, y')] & = & 
(\ell(u))[y]\cdot (\ell(u'))[y']-(\ell(u))[y']\cdot (\ell(u'))[y] \\
& = & u[f^{-1}(y)]\cdot u'[f^{-1}(y')]-u[f^{-1}(y')]\cdot u'[f^{-1}(y)].
\end{eqnarray*}
\item Suppose one of $y, y'$ are not in $\im(f)$. Then on the one hand, 
$\big(\tilde\ell(\phi_{G}(u, u'))\big)[(y, y')]=0$. On the other hand, 
$(\phi_{H}(\ell(u), \ell(u')))[(y, y')]$ is also $0$, as either 
$(\ell(u))[y]=(\ell(u'))[y]=0$, or 
$(\ell(u))[y']= (\ell(u'))[y']=0$.
\end{itemize}

We then verify that $\BLTfunctor$ forms a functor from the category of graphs 
with pullback homomorphisms to the category of $p$-groups of class $2$ and 
exponent $p$. For this, we need to verify that for two pullback homomorphisms 
$f, g: X\to Y$, $\BLTfunctor$ satisfies that $\BLTfunctor(f\circ 
g)=\BLTfunctor(f)\circ \BLTfunctor(g)$. To see this, we only need to 
verify that the composition of partial injective functions is preserved in the 
composition of linear maps, and this is straightforward. 

\paragraph{Algorithmic aspects of pullback homomorphisms.} Let $G=(X, E)$ and 
$H=(Y, F)$ be 
two graphs. Suppose $f:X\to Y$ is a pullback homomorphism. Let $K=H[\im(f)]$. We 
define the order (resp. size) of $f$ as the order (resp. size) of $K$. 
We can rephrase several classical algorithmic problem in terms of pullback 
homomorphisms by restricting to various classes of $G$, $H$, and $f$. 
\begin{itemize}
\item Suppose $H$ is the graph consisting of a perfect matching of large enough 
order. Then 
asking to maximize the size over all pullback homomorphisms is the maximum 
matching problem on $G$. 
\item Suppose $H$ is the complete graph of large enough order. Then asking to 
maximize the order over 
all pullback homomorphisms is the maximum clique problem on $G$. 
\item Suppose $G$ is the empty graph of large enough order. Then asking to 
maximize the order over all pullback homomorphisms is the independent set problem 
on $H$.
\item If $f$ is surjective, then the corresponding algorithmic problem is the 
well-known subgraph isomorphism problem, asking whether $H$ is isomorphic to a 
subgraph of $G$. 
\end{itemize}

\paragraph{Acknowledgement.} We thank Xiaoming Sun and Jialin Zhang for initial 
discussions which lead to the proof of Theorem~\ref{thm:iso}.

\bibliographystyle{alpha}
\bibliography{references}

\end{document}